\documentclass[10pt,draft]{article}
\usepackage{amsmath,amscd}
\usepackage{amssymb,latexsym,amsthm}
\usepackage{color}
\usepackage[spanish,english]{babel}
\usepackage{amssymb}
\usepackage{color}
\usepackage{amsmath,amsthm,amscd}
\usepackage[latin1]{inputenc}
\usepackage{lscape}
\usepackage{fancyhdr}
\usepackage{amsfonts}
\usepackage{pb-diagram}

\numberwithin{equation}{section}
\newtheorem{theorem}{Theorem}[section]

\newtheorem{definition}[theorem]{Definition}
\newtheorem{proposition}[theorem]{Proposition}
\newtheorem{lemma}[theorem]{Lemma}

\newtheorem{corollary}[theorem]{Corollary}

\theoremstyle{definition}

\newtheorem{example}[theorem]{Example}

\newtheorem{remark}[theorem]{Remark}

\newcommand{\cG}{\mbox{${\cal G}$}}

\newcommand{\cO}{\mbox{${\cal O}$}}

\newcommand{\cU}{\mbox{${\cal U}$}}

\newcommand{\cW}{\mbox{${\cal W}$}}

\title{\textbf{$d$-Hermite rings and skew $PBW$ extensions}}
\author{Oswaldo Lezama\\
\texttt{jolezamas@unal.edu.co}
\\Claudia Gallego
\\ Seminario de Álgebra Constructiva - SAC$^2$\\ Departamento de Matemáticas\\ Universidad Nacional de
Colombia, Sede Bogot\'a}
\date{}
\begin{document}
\maketitle
\begin{abstract}
\noindent Many rings and algebras arising in quantum mechanics can be interpreted as skew $PBW$
(Poincaré-Birkhoff-Witt) extensions. Indeed, Weyl algebras, enveloping algebras of
finite-dimensional Lie algebras (and its quantization), Artamonov quantum polynomials, diffusion
algebras, Manin algebra of quantum matrices, among many others, are examples of skew $PBW$
extensions. In this short paper we study the $d$-Hermite condition about stably free modules for
skew $PBW$ extensions. For this purpose, we estimate the stable rank of these non-commutative
rings. In addition, and close related with these questions, we will prove the Kronecker's theorem
about the radical of finitely generated ideals for some particular types of skew $BW$ extensions.
\bigskip

\noindent \textit{Key words and phrases.} $d$-Hermite rings, skew $PBW$ extensions, stable range
theorem, Kronecker's theorem.

\bigskip

\noindent 2010 \textit{Mathematics Subject Classification.} Primary: 16D40. Secondary: 15A21.
\end{abstract}
\section{Introduction}

Skew $PBW$ extensions are a class of non-commutative rings and algebras of polynomial type that
generalize classical $PBW$ extensions and include many important types of quantum algebras. Skew
$PBW$ extensions were defined in \cite{LezamaGallego}, and some homological properties of them were
investigated in \cite{lezamareyes1}. In particular, if the ring of coefficients satisfies some
suitable conditions, then the finitely generated projective modules over a skew $PBW$ extension are
stably free. However, it is easy to present examples of skew $PBW$ extensions that are not Hermite
rings (a ring $S$ is \textit{Hermite} if every stably free module is free). In fact, if $K$ is a
division ring, then $S:=K[x,y]$ is a trivial skew $PBW$ extension that has a module $M$ such that
$M\oplus S\cong S^2$, but $M$ is not free, i.e., $S$ is not Hermite (\cite{Lam}, p. 36). Another
example occurs in Weyl algebras: Let $K$ be a field, with ${\rm char}(K)=0$, the Weyl algebra
$A_1(K)=K[t][x;\frac{d}{dt}]$ is a skew $PBW$ extension but is not Hermite since there exist stably
free modules of rank $1$ over $A_n(K)$ that are not free (\cite{Cohn1}, Corollary 1.5.3; see also
\cite{McConnell}, Example 11.1.4). In this paper we will study a weaker condition than the Hermite
property for skew $PBW$ extensions: the $d$-Hermite condition.

Some notations and well known elementary properties of linear algebra for left modules are needed
in the rest of the paper. The reader can see also \cite{LezamaGallego1}. If nothing contrary is
assumed, all modules in this paper are left modules. Let $S$ be a ring, $S$ satisfies the
\textit{rank condition} {\rm(}$\mathcal{RC}${\rm)} if for any integers $r,s\geq 1$, given an
epimorphism $S^r\xrightarrow{f} S^s$, then $r\geq s$. $S$ is an $\mathcal{IBN}$ ring
{\rm(}\textit{invariant basis number}{\rm)} if for any integers $r,s\geq 1$, $S^r\cong S^s$ if and
only if $r=s$. It is well known that $\mathcal{RC}$ implies $\mathcal{IBN}$. From now on we will
assume that all rings considered in the present paper are $\mathcal{RC}$. Let $M$ be an $S$-module
and $t\geq 0$ an integer. $M$ is \textit{stably free of rank $t\geq 0$} if there exist an integer
$s\geq 0$ such that $S^{s+t}\cong S^s\oplus M$. It says that a ring $S$ is \textit{Hermite},
property denoted by $\mathcal{H}$, if every stably free $S$-module is free. Let $F$ be a matrix
over $S$ of size $r\times s$. Then
\begin{enumerate}
\item[\rm (i)]For $r\geq s$, $F$ is \textit{unimodular} if and only if $F$ has a left inverse.
\item[\rm (ii)]For $s\geq r$, $F$ is \textit{unimodular} if and only if $F$ has a right inverse.
\end{enumerate}
The set of unimodular column matrices of size $r\times 1$ is denoted by $Um_{c}(r,S)$. $Um_r(s,S)$
is the set of unimodular row matrices of size $1\times s$. The unimodular vector
$\textbf{\textit{v}}:=\begin{bmatrix}v_1 & \dots & v_r\end{bmatrix}^T\in Um_c(r,S)$ is called
\textit{stable} $($\textit{reducible}$)$ if there exists $a_1,\dots,a_{r-1}\in S$ such that
$\textbf{\textit{v}}':=\begin{bmatrix}v_1+a_1v_r & \dots & v_{r-1}+a_{r-1}v_r\end{bmatrix}^T$ is
unimodular. It says that the \textit{left stable rank} of $S$ is $d\geq 1$, denoted ${\rm
sr}(S)=d$, if $d$ is the least positive integer such that every unimodular column vector of length
$d+1$ is stable. It says that ${\rm sr}(S)=\infty$ if for every $d\geq 1$ there exits a non stable
unimodular column vector of length $d+1$. In a similar way is defined the right stable rank and it
is well known that the stable rank condition is left-right symmetric.

We conclude this preliminary section with a result due Stafford about stably free modules. A matrix
constructive proof can be found in \cite{Quadrat}, and also in \cite{LezamaGallego1}.

\begin{proposition}\label{7.3.6}
Let $S$ be a ring. Then any stably free $S$-module $M$ with ${\rm rank}(M)\geq \text{\rm sr}(S)$ is
free with dimension equals to ${\rm rank}(M)$.
\end{proposition}

\section{Skew $PBW$ extensions}\label{definitionexamplesspbw}

In this section we recall the definition of skew $PBW$ (Poincaré-Birkhoff-Witt) extensions defined
firstly in \cite{LezamaGallego}, and we will review also some basic properties about the polynomial
interpretation of this kind of non-commutative rings. Two particular subclasses of these extensions
are recalled also.
\begin{definition}\label{gpbwextension}
Let $R$ and $A$ be rings. We say that $A$ is a \textit{skew $PBW$ extension of $R$} $($also called
a $\sigma-PBW$ extension of $R$$)$ if the following conditions hold:
\begin{enumerate}
\item[\rm (i)]$R\subseteq A$.
\item[\rm (ii)]There exist finite elements $x_1,\dots ,x_n\in A$ such $A$ is a left $R$-free module with basis
\begin{center}
${\rm Mon}(A):= \{x^{\alpha}=x_1^{\alpha_1}\cdots x_n^{\alpha_n}\mid \alpha=(\alpha_1,\dots
,\alpha_n)\in \mathbb{N}^n\}$.
\end{center}
\item[\rm (iii)]For every $1\leq i\leq n$ and $r\in R-\{0\}$ there exists $c_{i,r}\in R-\{0\}$ such that
\begin{equation}\label{sigmadefinicion1}
x_ir-c_{i,r}x_i\in R.
\end{equation}
\item[\rm (iv)]For every $1\leq i,j\leq n$ there exists $c_{i,j}\in R-\{0\}$ such that
\begin{equation}\label{sigmadefinicion2}
x_jx_i-c_{i,j}x_ix_j\in R+Rx_1+\cdots +Rx_n.
\end{equation}
Under these conditions we will write $A:=\sigma(R)\langle x_1,\dots ,x_n\rangle$.
\end{enumerate}
\end{definition}
The following proposition justifies the notation and the alternative name given for the skew $PBW$
extensions.
\begin{proposition}\label{sigmadefinition}
Let $A$ be a skew $PBW$ extension of $R$. Then, for every $1\leq i\leq n$, there exists an
injective ring endomorphism $\sigma_i:R\rightarrow R$ and a $\sigma_i$-derivation
$\delta_i:R\rightarrow R$ such that
\begin{center}
$x_ir=\sigma_i(r)x_i+\delta_i(r)$,
\end{center}
for each $r\in R$.
\end{proposition}
\begin{proof}
See \cite{LezamaGallego}, Proposition 3.
\end{proof}

A particular case of skew $PBW$ extension is when all derivations $\delta_i$ are zero. Another
interesting case is when all $\sigma_i$ are bijective and the constants $c_{ij}$ are invertible. We
recall the following definition (cf. \cite{LezamaGallego}).
\begin{definition}\label{sigmapbwderivationtype}
Let $A$ be a skew $PBW$ extension.
\begin{enumerate}
\item[\rm (a)]
$A$ is quasi-commutative if the conditions {\rm(}iii{\rm)} and {\rm(}iv{\rm)} in Definition
\ref{gpbwextension} are replaced by
\begin{enumerate}
\item[\rm (iii')]For every $1\leq i\leq n$ and $r\in R-\{0\}$ there exists $c_{i,r}\in R-\{0\}$ such that
\begin{equation}
x_ir=c_{i,r}x_i.
\end{equation}
\item[\rm (iv')]For every $1\leq i,j\leq n$ there exists $c_{i,j}\in R-\{0\}$ such that
\begin{equation}
x_jx_i=c_{i,j}x_ix_j.
\end{equation}
\end{enumerate}
\item[\rm (b)]$A$ is bijective if $\sigma_i$ is bijective for
every $1\leq i\leq n$ and $c_{i,j}$ is invertible for any $1\leq i<j\leq n$.
\end{enumerate}
\end{definition}

Some useful properties of skew $PBW$ extensions that we will use later are the following.
\begin{proposition}\label{1.1.10}
Let A be a skew PBW extension of a ring R. If R is a domain, then A is a domain.
\end{proposition}
\begin{proof}
See \cite{lezamareyes1}.
\end{proof}

\begin{proposition}\label{1.3.3}
Let $A$ be a quasi-commutative skew $PBW$ extension of a ring $R$. Then,
\begin{enumerate}
\item[\rm (i)] $A$ is isomorphic to an iterated skew polynomial ring of
endomorphism type, i.e.,
\begin{center}
$A\cong R[z_1;\theta_1]\cdots [z_{n};\theta_n]$.
\end{center}
\item[\rm (ii)] If $A$ is bijective, then each
endomorphism $\theta_i$ is bijective, $1\leq i\leq n$.
\end{enumerate}
\end{proposition}
\begin{proof}
 See \cite{lezamareyes1}.
\end{proof}

\begin{proposition}\label{1.3.2}
Let $A$ be an arbitrary skew $PBW$ extension of $R$. Then, $A$ is a filtered ring with filtration
given by
\begin{equation}\label{eq1.3.1a}
F_m:=\begin{cases} R & {\rm if}\ \ m=0\\ \{f\in A\mid {\rm deg}(f)\le m\} & {\rm if}\ \ m\ge 1
\end{cases}
\end{equation}
and the corresponding graded ring $Gr(A)$ is a quasi-commutative skew $PBW$ extension of $R$.
Moreover, if $A$ is bijective, then $Gr(A)$ is a quasi-commutative bijective skew $PBW$ extension
of $R$.
\end{proposition}
\begin{proof}
See \cite{lezamareyes1}.
\end{proof}

\begin{proposition}[Hilbert Basis Theorem]\label{1.3.4}
Let $A$ be a bijective skew $PBW$ extension of $R$. If $R$ is a left $($right$)$ Noetherian ring
then $A$ is also a left $($right$)$ Noetherian ring.
\end{proposition}
\begin{proof}
See \cite{lezamareyes1}.
\end{proof}
Since the objects studied in the present paper are the skew $PBW$ extensions, it is necessary to
guarantee the $\mathcal{IBN}$ and $\mathcal{RC}$ properties for these rings.
\begin{lemma}\label{7.3.1}
Let $B$ be a filtered ring. If $Gr(B)$ is $\mathcal{RC}$ $($$\mathcal{IBN}$$)$, then $B$ is
$\mathcal{RC}$ $($$\mathcal{IBN}$$)$.
\end{lemma}
\begin{proof}
Let $\{B_p\}_{p\geq 0}$ be the filtration of $B$ and $f:B^r\to B^s$ an epimorphism. For $M:=B^r$ we
consider the standard positive filtration given by
\begin{center}
$F_0(M):=B_0\cdot e_1+\cdots+ B_0\cdot e_r$, $F_p(M):=B_pF_0(M)$, $p\geq 1$,
\end{center}
where $\{e_i\}_{i=1}^r$ is the canonical basis of $B^r$. Let $e_i':=f(e_i)$, then $B^s$ is
generated by $\{e_i'\}_{i=1}^r$ and $N:=B^s$ has an standard positive filtration given by
\begin{center}
$F_0(N):=B_0\cdot e_1'+\cdots+ B_0\cdot e_r'$, $F_p(N):=B_pF_0(N)$, $p\geq 1$.
\end{center}
Note that $f$ is filtered and strict: In fact, $f(F_p(M))=B_pf(F_0(M))=B_p(B_0\cdot f(e_1)+\cdots
+B_0\cdot f(e_r))=B_p(B_0\cdot e_1'+\cdots+ B_0\cdot e_r')=B_pF_0(N)=F_p(N)$. This implies that
$Gr(M)\xrightarrow{Gr(f)}Gr(N)$ is surjective. If we prove that $Gr(M)$ and $Gr(N)$ are free over
$Gr(B)$ with bases of $r$ and $s$ elements, respectively, then from the hypothesis we conclude that
$r\geq s$ and hence $B$ is $\mathcal{RC}$.

Since every $e_i\in F_0(M)$ and $F_p(M)=\sum_{i=1}^r \oplus B_p\cdot e_i$, $M$ is filtered-free
with filtered-basis $\{e_i\}_{i=1}^r$, so $Gr(M)$ is graded-free with graded-basis
$\{\overline{e_i}\}_{i=1}^r$, $\overline{e_i}:=e_i+F_{-1}(M)=e_i$ (recall that by definition of
positive filtration, $F_{-1}(M):=0$). For $Gr(N)$ note that $N$ is also filtered-free with respect
the filtration $\{F_p(N)\}_{p\geq 0}$ given above: Indeed, we will show next that the canonical
basis $\{f_j\}_{j=1}^s$ of $N$ is a filtered basis. If $f_j=x_{j1}\cdot e_1'+\cdots+x_{jr}\cdot
e_r'$, with $x_{ji}\in B_{p_{ij}}$, let $p:=\max\{p_{ij}\}$, $1\leq i\leq r$, $1\leq j\leq s$, then
$f_j\in F_p(N)$, moreover, for every $q$, $B_{q-p}\cdot f_1\oplus \cdots \oplus B_{q-p}\cdot
f_s\subseteq B_{q-p}F_p(N)\subseteq F_q(N)$ (recall that for $k<0$, $B_k=0$); in turn, let $x\in
F_q(N)$, then $x=b_1\cdot f_1+\cdots+b_s\cdot f_s$ and in $Gr(N)$ we have $\overline{x}\in
Gr(N)_q$, $\overline{x}=\overline{b_1}\cdot \overline{f_1}+\cdots+\overline{b_s}\cdot
\overline{f_s}$, if $b_j\in B_{u_j}$, let $u:=\max\{u_{j}\}$, so $\overline{b_j}\cdot
\overline{f_j}\in Gr(N)_{u+p}$, so $q=u+p$, i.e., $u=q-p$ and hence $x\in B_{q-p}\cdot f_1\oplus
\cdots \oplus B_{q-p}\cdot f_s$, Thus, we have proved that $B_{q-p}\cdot f_1\oplus \cdots \oplus
B_{q-p}\cdot f_s=F_q(N)$, for every $q$, and consequently, $\{f_j\}_{j=1}^s$ is a filtered basis of
$N$. From this we conclude that $Gr(N)$ is graded-free with graded-basis
$\{\overline{f_j}\}_{j=1}^s$, $\overline{f_j}:=f_j+F_{p-1}(N)$.

We can repeat the previuos proof for the $\mathcal{IBN}$ property but assuming that $f$ is an
isomorphism.
\end{proof}

\begin{theorem}\label{618}
Let $A$ be a skew $PBW$ extension of a ring $R$. Then, $A$ is $\mathcal{RC}$ $(\mathcal{IBN})$ if
and only if $R$ is $\mathcal{RC}$ $(\mathcal{IBN})$.
\end{theorem}
\begin{proof}
We consider only the proof for $\mathcal{RC}$, the case $\mathcal{IBN}$ is completely analogous.

$\Rightarrow)$: Since $R\hookrightarrow A$, if $A$ is $\mathcal{RC}$, then $R$ is $\mathcal{RC}$.
In fact, let $S$ and $T$ be rings and let $S\xrightarrow{f} T$ be a ring homomorphism, if $T$ is a
$\mathcal{RC}$ ring then $S$ is also a $\mathcal{RC}$ ring: $T$ is a right $S$-module, $t\cdot
s:=tf(s)$; suppose that $S^r\xrightarrow{f}S^s$ is an epimorphism, then $T\otimes_S
S^r\xrightarrow{i_T\otimes f}T\otimes_S S^s$ is also an epimorphism of left $T$-modules, i.e., we
have an epimorphism $T^r\rightarrow T^s$, so $r\geq s$

$\Leftarrow)$: We consider first the skew polynomial ring $R[x;\sigma]$ of endomorphism type, then
$R[x;\sigma]\to R$ given by $p(x)\to p(0)$ is a ring homomorphism, so $R[x;\sigma]$ is
$\mathcal{RC}$ since $R$ is $\mathcal{RC}$. By Propositions \ref{1.3.3} and \ref{1.3.2}, $Gr(A)$ is
isomorphic to an iterated skew polynomial ring $R[z_1;\theta_1]\cdots [z_{n};\theta_n]$, so $Gr(A)$
is $\mathcal{RC}$. Only rest to apply Lemma \ref{7.3.1}.
\end{proof}

\section{$d$-Hermite rings and stable rank}

There is a famous conjecture in commutative algebra that says that if $R$ is a commutative
$\mathcal{H}$-ring, then the polynomial ring $R[x]$ is $\mathcal{H}$ (see \cite{Lam}). As we
observed at the beginning of the paper, this conjecture for skew $PBW$ extensions is not true.
Another example is the skew polynomial ring $K[t][x;\sigma]$, with $K$ a field and
$\sigma(t):=t+1$; in \cite{McConnell} is proved that $K[t][x;\sigma]$ is not $\mathcal{H}$ although $K[t]$ is
$\mathcal{H}$. Thus, instead of considering the $\mathcal{H}$ condition and the conjecture for skew
$PBW$ extensions, we will study a weakly property, the $d$-Hermite property. The following
proposition induces the definition of $d$-Hermite rings.
\begin{proposition}\label{7.1.1}
Let $S$ be a ring. For any integer $d\geq 0$, the following statements are equivalent:
\begin{enumerate}
\item[{\rm (i)}]Any stably free module of rank $\geq d$ is free.
\item[{\rm (ii)}]Any unimodular row matrix over $S$ of length $\geq d+1$ can be
completed to an invertible matrix over $S$.
\item[{\rm (iii)}]For every $r\geq d+1$, if $\textbf{u}$ is an unimodular row matrix of size $1\times
r$, then there exists an invertible matrix $U\in GL_{r}(S)$ such that $\textbf{u}U=(1,0,\dots, 0)$,
i.e., the general linear group $GL_r(S)$ acts transitively on $Um_r(r,S)$.
\item[$\rm (iv)$]For every $r\geq d+1$, given an unimodular matrix $F$ of size $s\times r$,
$r\geq s$, there exists $U\in GL_r(S)$ such that
\begin{center}
$FU= \begin{bmatrix}I_s & | & 0\end{bmatrix}$.
\end{center}
\end{enumerate}
\end{proposition}
\begin{proof}
We can repeat the proof of Theorem 2 in \cite{LezamaGallego1} taking $r\geq d+1$.
\end{proof}
\begin{definition}\label{7.1.2}
Let $S$ be a ring and $d\geq 0$ an integer. $S$ is $d$-Hermite, property denoted by
$d$-$\mathcal{H}$, if $S$ satisfies any of conditions in Proposition \ref{7.1.1}.
\end{definition}

\begin{proposition}
The $d$-Hermite condition is left-right symmetric.
\end{proposition}
\begin{proof}
We can repeat the proof of Proposition 9 in \cite{LezamaGallego1} taking $r\geq d+1$. See also
\cite{McConnell}, Lemma 11.1.13.
\end{proof}
\begin{remark}\label{7.1.5}
(i) Observe that $0$-Hermite rings coincide with $\mathcal{H}$ rings, and for commutative rings,
$1$-Hermite coincides also with $\mathcal{H}$ (see \cite{Lam}, Theorem I.4.11). If $K$ is a field
with ${\rm char}(K)=0$, by Corollary 4 in \cite{LezamaGallego1}, $A_1(K)$ is $2$-$\mathcal{H}$ but,
as we observed at the beginning of the chapter, $A_1(K)$ is not $1$-$\mathcal{H}$. In general,
$\mathcal{H}\subsetneq 1$-$\mathcal{H}\subsetneq 2$-$\mathcal{H}\subsetneq \cdots$ (see
\cite{Cohn1}).

(ii) Note that $\mathcal{H}=1$-$\mathcal{H}$$\cap \mathcal{WF}$ (a ring $S$ is $\mathcal{WF}$,
\textit{weakly finite}, if for all $n\geq 0$, $P\oplus S^n\cong S^n$ if and only if $P=0$).
\end{remark}

\begin{proposition}\label{7.1.4}
Let $S$ be a ring. Then, $S$ is ${\rm sr}(S)$-$\mathcal{H}$.
\end{proposition}
\begin{proof}
This follows from Definition \ref{7.1.2} and Theorem \ref{7.3.6}.
\end{proof}
\begin{corollary}\label{7.1.4b}
Let $S$ be a ring. If ${\rm sr}(S)=1$, then $S$ is $\mathcal{H}$.
\end{corollary}
\begin{proof}
According to Proposition \ref{7.1.4}, $S$ is $1$-$\mathcal{H}$, however, it is well known that
rings with stable rank $1$ are cancellable (see \cite{Evans}), so by proposition 12 in
\cite{LezamaGallego1}, $S$ is $\mathcal{H}$.
\end{proof}

Proposition \ref{7.1.4} motivates the task of computing the stable rank of skew $PBW$ extensions.
For this purpose we need to recall the famous stable range theorem. This theorem relates the stable
rank and the Krull dimension of a ring. The original version of this classical result is due a Bass
(1968, \cite{Bass2}) and states that if $S$ is a commutative Noetherian ring and ${\rm Kdim}(S)=d$
then ${\rm sr}(S)\leq d+1$. Heitmann extends the theorem for arbitrary commutative rings (1984,
\cite{Heitmann}). Lombardi et. al. in 2004 (\cite{Lombardi}, Theorem 2.4; see also
\cite{Lombardi3}) proved again the theorem for arbitrary commutative rings using the Zariski
lattice of a ring and the boundary ideal of an element. This proof is elementary and constructive.
Stafford in 1981 (\cite{Stafford3}) proved a noncommutative version of the theorem for left
Noetherian rings.

\begin{proposition}[Stable range theorem]\label{821}
Let $S$ be  a left Noetherian ring and ${\rm lKdim}(S)=d$, then ${\rm sr}(S)\leq d+1$.
\end{proposition}
\begin{proof}
See \cite{Stafford3}.
\end{proof}

From this we get the following result.

\begin{theorem}\label{7.2.1}
Let $R$ be a left Noetherian ring with finite left Krull dimension and $A=\sigma(R)\langle
x_1,\dots,x_n\rangle$ a bijective skew $PBW$ extension of $R$, then
\begin{center}
$1\leq {\rm sr}(A)\leq {\rm lKdim}(R)+n+1$,
\end{center}
and $A$ is $d$-$\mathcal{H}$, with $d:=({\rm lKdim}(R)+n+1)$.
\end{theorem}
\begin{proof}
The inequalities follow from Proposition \ref{821} and Theorem 4.2 in \cite{lezamareyes1}. The
second statement follows from Proposition \ref{7.1.4}.
\end{proof}

\begin{example}
The results in \cite{lezamareyes1} for the Krull dimension of bijective skew $PBW$ extensions can
be combined with Theorem \ref{7.2.1} in order to get an upper bound for the stable rank. With this
we can estimate also the $d$-Hermite condition. The next table gives such estimations:

\begin{table}[htb]\label{table8.1}
\centering \tiny{
\begin{tabular}{|l|l|}\hline 
\textbf{Ring} & \textbf{U. B.}\\
\hline
\cline{1-2} Habitual polynomial ring $R[x_1,\dotsc,x_n]$ & ${\rm dim}(R)+n+1$ \\
\cline{1-2} Ore extension of bijective type $R[x_1;\sigma_1 ,\delta_1]\cdots [x_n;\sigma_n,\delta_n]$  & ${\rm dim}(R)+n+1$ \\
\cline{1-2} Weyl algebra $A_n(K)$ & $2n+1$ \\
\cline{1-2} Extended Weyl algebra $B_n(K)$ & $n+1$ \\
\cline{1-2} Universal enveloping algebra of a Lie algebra $\mathfrak{g}$, $\cU(\mathfrak{g})$, $K$ commutative ring & ${\rm dim}(K)+n+1$ \\
\cline{1-2} Tensor product $R\otimes_K \cU(\cG)$ & ${\rm dim}(R)+n+1$ \\
\cline{1-2} Crossed product $R*\cU(\cG)$ & ${\rm dim}(R)+n+1$ \\
\cline{1-2} Algebra of q-differential operators $D_{q,h}[x,y]$ & $3$ \\
\cline{1-2} Algebra of shift operators $S_h$ & $3$ \\
\cline{1-2} Mixed algebra $D_h$ & $4$ \\
\cline{1-2} Discrete linear systems $K[t_1,\dotsc,t_n][x_1,\sigma_1]\dotsb[x_n;\sigma_n]$ & $2n+1$ \\
\cline{1-2} Linear partial shift operators $K[t_1,\dotsc,t_n][E_1,\dotsc,E_n]$ & $2n+1$ \\
\cline{1-2} Linear partial shift operators $K(t_1,\dotsc,t_n)[E_1,\dotsc,E_n]$ & $n+1$ \\
\cline{1-2} L. P. Differential operators $K[t_1,\dotsc,t_n][\partial_1,\dotsc,\partial_n]$ & $2n+1$ \\
\cline{1-2} L. P. Differential operators $K(t_1,\dotsc,t_n)[\partial_1,\dotsc,\partial_n]$ & $n+1$ \\
\cline{1-2} L. P. Difference operators $K[t_1,\dotsc,t_n][\Delta_1,\dotsc,\Delta_n]$ & $2n+1$ \\
\cline{1-2} L. P. Difference operators $K(t_1,\dotsc,t_n)[\Delta_1,\dotsc,\Delta_n]$ & $n+1$ \\
\cline{1-2} L. P. $q$-dilation operators $K[t_1,\dotsc,t_n][H_1^{(q)},\dotsc,H_m^{(q)}]$ & $n+m+1$ \\
\cline{1-2} L. P. $q$-dilation operators $K(t_1,\dotsc,t_n)[H_1^{(q)},\dotsc,H_m^{(q)}]$ & $m+1$ \\
\cline{1-2} L. P. $q$-differential operators $K[t_1,\dotsc,t_n][D_1^{(q)},\dotsc,D_m^{(q)}]$ & $n+m+1$ \\
\cline{1-2} L. P. $q$-differential operators $K(t_1,\dotsc,t_n)[D_1^{(q)},\dotsc,D_m^{(q)}]$ & $m+1$ \\
\cline{1-2} Diffusion algebras & $2n+1$ \\
\cline{1-2} Additive analogue of the Weyl algebra $A_n(q_1,\dotsc,q_n)$ & $2n+1$ \\
\cline{1-2} Multiplicative analogue of the Weyl algebra $\cO_n(\lambda_{ji})$ & $n+1$ \\
\cline{1-2} Quantum algebra $\cU'(\mathfrak{so}(3,K))$ & $4$ \\
\cline{1-2} 3-dimensional skew polynomial algebras & $4$ \\
\cline{1-2} Dispin algebra $\cU(osp(1,2))$ & $4$ \\
\cline{1-2} Woronowicz algebra $\cW_{\nu}(\mathfrak{sl}(2,K))$ & $4$ \\
\cline{1-2} Complex algebra $V_q(\mathfrak{sl}_3(\mathbb{C}))$ & $11$ \\
\cline{1-2} Algebra \textbf{U} & $3n+1$ \\
\cline{1-2} Manin algebra $\cO_q(M_2(K))$ & $5$ \\
\cline{1-2} Coordinate algebra of the quantum group $SL_q(2)$ & $5$ \\
\cline{1-2} $q$-Heisenberg algebra \textbf{H}$_n(q)$ & $3n+1$ \\
\cline{1-2} Quantum enveloping algebra of $\mathfrak{sl}(2,K)$, $\cU_q(\mathfrak{sl}(2,K))$ & $4$ \\
\cline{1-2} Hayashi algebra $W_q(J)$ & $3n+1$\\
\cline{1-2} Differential operators on a quantum space $S_{\textbf{q}}$,
$D_{\textbf{q}}(S_{\textbf{q}})$ & $2n+1$ \\
\cline{1-2} Witten's Deformation of $\cU(\mathfrak{sl}(2,K)$ & $4$ \\
\cline{1-2} Quantum Weyl algebra of Maltsiniotis $A_n^{\textbf{q},\lambda}$, $K$ commutative ring & ${\rm dim}(K)+2n+1$\\
\cline{1-2} Quantum Weyl algebra $A_n(q,p_{i,j})$ & $2n+1$\\
\cline{1-2} Multiparameter Weyl algebra $A_n^{Q,\Gamma}(K)$ & $2n+1$\\
\cline{1-2} Quantum symplectic space $\cO_q(\mathfrak{sp}(K^{2n}))$ & $2n+1$ \\
\cline{1-2} Quadratic algebras in $3$ variables & $4$ \\
\hline
\end{tabular}}\label{table8.1}
\caption{Stable rank for some examples of bijective skew $PBW$ extensions.}\label{table8.1}
\end{table}
\end{example}

\section{Kronecker's theorem}

Close related to the stable range theorem is the Kronecker's theorem staying that if $S$ is a
commutative ring with ${\rm Kdim} (S)<d$, then every finitely generated ideal $I$ of $S$ has the
same radical as an ideal generated by $d$ elements. In this section we want to investigate this
theorem for noncommutative rings using the Zariski lattice and the boundary ideal, but generalizing
these tools and its properties to noncommutative rings. The main result will be applied to skew
$PBW$ extensions.

\begin{definition}\label{7.2.1e}
Let $S$ be a ring and $Spec(S)$ the set of all prime ideals of $S$. The Zariski lattice of $S$ is
defined by
\[
Zar(S):=\{D(X)|X\subseteq S\}, \ \text{with}\ D(X):=\bigcap_{X\subseteq P\in Spec(S)}P.
\]
\end{definition}
\noindent $Zar(S)$ is ordered with respect the inclusion. The description of the Zariski lattice is
presented in the next proposition, $\langle X\}, \langle X\rangle, \{X\rangle$ will represents the
left, two-sided, and right ideal of $S$ generated by $X$, respectively. $\vee$ denotes the $\sup$
and $\wedge$ the $\inf$.
\begin{proposition}\label{7.2.1d}
Let $S$ be a ring, $I,I_1,I_2, I_3$ two-sided ideals of $S$, $X\subseteq S$, and
$x_1,\dots,x_n,x,y\in S$. Then,
\begin{enumerate}
\item[\rm (i)]$D(X)=D(\langle X\})=D(\langle
X\rangle)=D(\{X \rangle)$.
\item[\rm (ii)]$D(I)=rad(S)$ if and only if $I\subseteq rad(S)$. In particular, $D(0)=rad(S)$.
\item[\rm (iii)]$D(I)=S$ if and only if $I=S$.
\item[\rm (iv)]$I\subseteq D(I)$ and $D(D(I))=D(I)$. Moreover, if $I_1\subseteq I_2$, then $D(I_1)\subseteq D(I_2)$.
\item[\rm (v)]Let $\{I_j\}_{j\in \mathcal{J}}$ a family of two-sided ideals of $S$. Then, $D(\sum_{j\in \mathcal{J}}I_j)=\vee_{j\in
\mathcal{J}}D(I_j)$. In particular, $D(x_1,\dots,x_n)=D(x_1)\vee\cdots \vee D(x_n)$.
\item[\rm (vi)]$D(I_1I_2)=D(I_1)\wedge D(I_2)$. In particular, $D(\langle x\rangle \langle y\rangle)=D(x)\wedge D(y)$.
\item[\rm (vii)]$D(x+y)\subseteq D(x,y)$.
\item[\rm (viii)]If $\langle x\rangle\langle y\rangle\subseteq D(0)$, then $D(x,y)=D(x+y)$.
\item[\rm (ix)]If $x\in D(I)$, then $D(I)=D(I,x)$.
\item[\rm (x)]If $\overline{S}:=S/I$, then $D(\overline{J})=\overline{D(J)}$, for any two-sided ideal $J$ of $S$ containing
$I$.
\item[\rm (xi)]$u\in D(I)$ if and only if $\overline{u}\in rad(S/I)$. In such case, if $u\in D(I)$, there exists $k\geq 1$ such that $u^k\in I$.
\item[\rm (xii)]$Zar(S)$ is distributive:
\begin{center}
$D(I_1)\wedge [D(I_2)\vee D(I_3)]=[D(I_1)\wedge D(I_2)]\vee [D(I_1)\wedge D(I_3)]$,

$D(I_1)\vee [D(I_2)\wedge D(I_3)]=[D(I_1)\vee D(I_2)]\wedge [D(_1)\vee D(I_3)]$.
\end{center}
\end{enumerate}
\end{proposition}
\begin{proof}
(i), (ii), (iv), (ix) and (x) are evident from the definitions.

(iii) If $I=S$ there is no prime ideal containing $I$, so the intersection of prime ideals
containing $I$  is taken equals $S$ (see \cite{Goodearl}, p. 51). Conversely, if $I\neq S$ the
intersection of proper ideals containing $I$ is proper (this collection is not empty since $I$ is
contained in at least one prime ideal), thus $D(I)\neq S$.

(v) We prove first that $\vee_{j\in \mathcal{J}}D(I_j)=D(\sum_{j\in \mathcal{J}}D(I_j))$: for every
$j\in \mathcal{J}$, $D(I_j)\subseteq \sum_{j\in \mathcal{J}}D(I_j)\subseteq D(\sum_{j\in
\mathcal{J}}D(I_j))$; let $D(I)\supseteq D(I_j)$ for every $j\in \mathcal{J}$, then $D(I)\supseteq
\sum_{j\in \mathcal{J}}D(I_j)$ and hence $D(I)=D(D(I))\supseteq D(\sum_{j\in \mathcal{J}}D(I_j))$.

Only rest to show that $D(\sum_{j\in \mathcal{J}}D(I_j))=D(\sum_{j\in \mathcal{J}}I_j)$: from
$I_j\subseteq \sum_{j\in \mathcal{J}}I_j$ we get $D(I_j)\subseteq D(\sum_{j\in \mathcal{J}}I_j)$,
so $D(\sum_{j\in \mathcal{J}}I_j)\supseteq \vee_{j\in \mathcal{J}}D(I_j)=D(\sum_{j\in
\mathcal{J}}D(I_j))$; on the other hand, $D(\sum_{j\in \mathcal{J}}D(I_j))\supseteq \sum_{j\in
\mathcal{J}}D(I_j)\supseteq \sum_{j\in \mathcal{J}}I_j$, so $D(D(\sum_{j\in
\mathcal{J}}D(I_j)))\supseteq D(\sum_{j\in \mathcal{J}}I_j)$, thus $D(\sum_{j\in
\mathcal{J}}D(I_j))\supseteq D(\sum_{j\in \mathcal{J}}I_j)$.

(vi) It is clear that $D(I_1I_2)\subseteq D(I_1),D(I_2)$. Let $I$ be a two-side ideal of $S$ such
that $D(I)\subseteq D(I_1),D(I_2)$, then $D(I)\subseteq D(I_1)\cap D(I_2)\subseteq D(I_1I_2)$. The
last inclusion follows from the fact that if $P$ is a prime ideal containing $I_1I_2$, then
$I_1\subseteq P$ or $I_2\subseteq P$, thus if $x\in D(I_1)\cap D(I_2)$, then $x\in P$, i.e., $x\in
D(I_1I_2)$. This implies that $D(I_1)\wedge D(I_2)=D(I_1I_2)$.

(vii) Since $\langle x+y\rangle\subseteq \langle x,y\rangle$, then the result follows from (iv).

(viii) According to (vii), $D(x+y)\subseteq D(x,y)$; for the other inclusion note first that
$D(x,y)=D(x+y,\langle x\rangle\langle y\rangle)$: the inclusion $D(x+y,\langle x\rangle\langle
y\rangle)\subseteq D(x,y)$ is clear since any prime containing $x,y$ contains $x+y,\langle
x\rangle\langle y\rangle$. Let $P$ a prime that contains $x+y,\langle x\rangle\langle y\rangle$, so
$x\in P$ or $y\in P$, in the first case $x\in P$ and $y\in P$ and the same is true in the second
case. This implies that $D(x,y)\subseteq D(x+y,\langle x\rangle\langle y\rangle)$.

By the hypothesis and numeral (ii), $\langle x\rangle\langle y\rangle\subseteq rad(S)$, i.e.,
$\langle x\rangle\langle y\rangle$ is contained in all primes, so $D(x+y,\langle x\rangle\langle
y\rangle)=D(x+y)$ and hence $D(x,y)=D(x+y)$.

(xi) The first assertion is clear from the definition of $D(I)$ and $rad(S/I)$. If $u\in D(I)$,
then $\overline{u}\in rad(S/I)$ and hence $\overline{u}$ is strongly nilpotent, but this implies
that $\overline{u}$ is nilpotent (see \cite{McConnell}), i.e., there exists $k\geq 1$ such that
$\overline{u}^{k}=\overline{0}$, i.e., $u^k\in I$.

(xii) For the first identity we have:
\begin{center}
$D(I_1)\wedge [D(I_2)\vee D(I_3)]=D(I_1)\wedge
D(I_2+I_3)=D[I_1(I_2+I_3)]=D(I_1I_2+I_1I_3)=D(I_1I_2)\vee D(I_1I_3)=[D(I_1)\wedge D(I_2)]\vee
[D(I_1)\wedge D(I_3]$.
\end{center}
For the second relation we have
\begin{center}
$D(I_1)\vee [D(_2)\wedge D(I_3)]=D(I_1)\vee D(I_2I_3)=D(I_1+I_2I_3)\supseteq
D[(I_1+I_2)(I_1+I_3)]=[D(I_1)\vee D(I_2)]\wedge [D(I_1)\vee D(I_3)]$;
\end{center}
the other inclusion follows from the fact that $D(I_1+I_2I_3)\subseteq D[(I_1+I_2)(I_1+I_3)]$ since
if $P$ is a prime ideal that contains $(I_1+I_2)(I_1+I_3)$, then $P\supseteq (I_1+I_2)$ or
$P\supseteq (I_1+I_3)$, thus $P\supseteq I_1$ and $P\supseteq I_2\supseteq I_2I_3$, or, $P\supseteq
I_1$ and $P\supseteq I_3\supseteq I_2I_3$, i.e., $P\supseteq I_1+I_2I_3$.
\end{proof}

\begin{definition}
Let $S$ be a ring and $v\in S$, the boundary ideal of $v$ is defined by $I_v:=\langle v
\rangle+(D(0):\langle v\rangle)$, where $(D(0):\langle v\rangle):=\{x\in S|\langle v\rangle
x\subseteq D(0)\}$.
\end{definition}
Note that $I_v\neq 0$ for every $v\in S$. On the other hand, if $v$ is invertible or if $v=0$, then
$I_v=S$. If $S$ a domain and $v\neq 0$, then $I_v=\langle v\rangle$.

\begin{definition}\label{7.2.1c}
Let $S$ be a ring such that ${\rm lKdim}(S)$ exists. We say the $S$ satisfies the boundary
condition if for any $d\geq 0$ and every $v\in S$,
\begin{center}
${\rm lKdim}(S)\leq d\Rightarrow{\rm lKdim}(S/I_v)\leq d-1$.
\end{center}
\end{definition}

\begin{example}\label{846}
(i) Any commutative Noetherian ring satisfies the boundary condition: Indeed, for commutative
Noetherian rings, the classical Krull dimension and the Krull dimension coincide, so we can apply
Theorem 13.2 in \cite{Lombardi3}.

(ii) Any prime ring $S$ with left Krull dimension satisfies the boundary condition: In fact, for
prime rings, any non-zero two sided ideal is essential, so ${\rm lKdim}(S/I_v)< {\rm lKdim}(S)$
(see \cite{McConnell}, Proposition 6.3.10).

(iii) Any domain with left Krull dimension satisfies the satisfies the boundary condition: Indeed,
any domain is a prime ring.
\end{example}

\begin{lemma}[Kronecker]\label{7.2.1b}
Let $S$ be a domain such that ${\rm lKdim}(S)$ exists. If ${\rm lKdim}(S)<d$ and
$u_1,\dots,u_{d},u\in S$, then there exist $x_1,\dots,x_d\in S$ such that
\begin{center}
$D(u_1,\dots,u_{d},u)=D(u_1+x_1u,\dots, u_d+x_du)$.
\end{center}
\end{lemma}
\begin{proof}
The proof is by induction on $d$. Let $d=1$ and $u_1,u\in S$, if ${\rm lKdim}(S)=-1$, then by
definition $S=0$ and $u_1,u=0$, so we take $x_1:=0$. Let ${\rm lKdim}(S)=0$; by the boundary
condition, ${\rm lKdim}(S/I_{u_1})=-1$, i.e., $S=I_{u_1}=\langle u_1\rangle+(D(0):\langle
u_1\rangle)$. There exist $c_1,c_1',\dots,c_l,c_l'\in S$  and $x_1\in (D(0):\langle u_1\rangle)$
such that $1=c_1u_1c_1'+\cdots+c_lu_1c_l'+x_1$, then $\langle u_1\rangle\langle x_1\rangle\subseteq
D(0)$ and $u=c_1u_1c_1'u+\cdots+c_lu_1c_l'u+x_1u$, thus $u\in \langle u_1,x_1u\rangle$ and hence
$u\in D(u_1,x_1u)$ (Proposition \ref{7.2.1d}, part (iv)). Moreover, $\langle u_1\rangle\langle
x_1u\rangle\subseteq D(0)$, then by Proposition \ref{7.2.1d}, part (viii),
$D(u_1,x_1u)=D(u_1+x_1u)$. Thus, $u\in D(u_1+x_1u)$, so $D(u_1+x_1u)=D(u_1+x_1u,u)$ (Proposition
\ref{7.2.1d}, part (ix)), but $D(u_1+x_1u,u)=D(u_1,u)$ since $\langle u_1+x_1u,u\rangle=\langle
u_1,u \rangle$, so $D(u_1,u)=D(u_1+x_1u)$.

Now we assume that the proposition is true for rings with left Krull dimension $<d-1$, $d\geq 2$,
and let $S$ be a ring with ${\rm lKdim}(S)<d$. Let $u_1,\dots,u_{d},u\in S$. We consider two cases.

\textit{Case 1}. If $u_d=0$, then the theorem is trivial with $x_1=\cdots=x_{d-1}=0$, $x_d=1$.

\textit{Case 2}. Let $u_d\neq 0$. Let $I$ be the boundary ideal of $u_{d}$, then $D(I)=\langle
u_d\rangle$. We consider the elements $\overline{u_1},\dots,\overline{u_{d-1}},\overline{u}\in
\overline{S}$, with $\overline{S}:=S/I$. By the hypothesis, ${\rm lKdim}(\overline{S})<d-1$ and
hence there exist elements $\overline{x_1},\dots,\overline{x_{d-1}}\in \overline{S}$ such that
$D(\overline{u_1},\dots,\overline{u_{d-1}},\overline{u})=D(\overline{u_1}+\overline{x_1}\,\overline{u},\dots,
\overline{u_{d-1}}+\overline{x_{d-1}}\,\overline{u})$. From this we get that
\begin{center}
$D(\overline{\langle u_1, \dots, u_{d-1},u\rangle+I})=D(\overline{\langle u_1+x_1u,\dots,
u_{d-1}+x_{d-1}u\rangle+I})$,
\end{center}
but by Proposition \ref{7.2.1d}, part (x),
\begin{center}
$\overline{D(\langle u_1, \dots, u_{d-1},u\rangle+I)}=\overline{D(\langle u_1+x_1u,\dots,
u_{d-1}+x_{d-1}u\rangle+I)}$, i.e.,
\end{center}
\begin{center}
$D(\langle u_1, \dots, u_{d-1},u\rangle+I)=D(\langle u_1+x_1u,\dots, u_{d-1}+x_{d-1}u\rangle+I)$.
\end{center}
Since $u\in \langle u_1, \dots, u_{d-1},u\rangle+I\subseteq D(\langle u_1, \dots,
u_{d-1},u\rangle+I)$, then $u\in D(\langle u_1+x_1u,\dots, u_{d-1}+x_{d-1}u\rangle+I)=D(\langle
u_1+x_1u,\dots, u_{d-1}+x_{d-1}u)\vee D(I)=D(\langle u_1+x_1u,\dots, u_{d-1}+x_{d-1}u, u_d)$.
Taking $x_d:=0$ we get that $u\in D(u_1+x_1u,\dots, u_{d-1}+x_{d-1}u,u_{d}+x_du)$. From this, and
using Proposition \ref{7.2.1d}, part (ix), we conclude that
\begin{center}
$D(u_1+x_1u,\dots, u_{d-1}+x_{d-1}u,u_{d}+x_du)=D(u_1+x_1u,\dots, u_{d-1}+x_{d-1}u,u_{d}+x_du,u)$
\end{center}
however note that
\begin{center}
$\langle u_1+x_1u,\dots, u_{d-1}+x_{d-1}u,u_{d}+x_du,u\rangle=\langle u_1,\dots,
u_{d-1},u_{d},u\rangle$,
\end{center}
so $D(u_1+x_1u,\dots, u_{d-1}+x_{d-1}u,u_{d}+x_du)=D(u_1,\dots, u_{d-1},u_{d},u)$.
\end{proof}

\begin{corollary}\label{7.2.1a}
Let $S$ be a domain such that ${\rm lKdim}(S)$ exists. If ${\rm lKdim}(S)<d$ and $u_1,\dots,
u_{d+1}\in S$ are such that $\langle u_1,\dots,u_{d+1}\rangle =S$, then there exist elements
$x_1,\dots, x_d\in S $ such that $\langle u_1+x_1u_{d+1}, \dots, u_d+x_du_{d+1}\rangle =S$.
\end{corollary}
\begin{proof}
The statement follows directly from Proposition \ref{7.2.1d}, part (iii), and Lemma \ref{7.2.1b}.
\end{proof}

\begin{theorem}
Let $A=\sigma(R)\langle x_1,\dots,x_n\rangle$ be a bijective skew $PBW$ extension of a left
Noetherian domain $R$. If ${\rm lKdim}(R)<d$ and $u_1,\dots,u_{d+n},u\in A$, then there exist
$y_1,\dots,y_{d+n}\in A$ such that
\begin{center}
$D(u_1,\dots,u_{d+n},u)=D(u_1+y_1u,\dots, u_{d+n}+y_{d+n}u)$.
\end{center}
\end{theorem}
\begin{proof}
This follows directly from Propositions \ref{1.1.10} and \ref{1.3.4}, Theorem 4.2 in
\cite{lezamareyes1}, and Lemma \ref{7.2.1b}.
\end{proof}



\end{document}